\documentclass[12pt,reqno]{amsart}
\usepackage{amscd}
\usepackage{amsmath}
\usepackage{amssymb}
\usepackage{units}
\usepackage{listings}
\usepackage{enumitem}
\usepackage{tikz}
\usetikzlibrary{shapes}
\usepackage[all]{xy}
\usepackage{algorithm}
\usepackage{algorithmic}

\usepackage{hyperref}
\usepackage{hypcap}

\def\frk{\frak}               

\def\Phi{{\frk n}}
\def\Phi{{\frk N}}
\def\opn#1#2{\def#1{\operatorname{#2}}} 
\opn\chara{char} \opn\length{\ell} \opn\pd{pd} \opn\rk{rk}
\opn\projdim{proj\,dim} \opn\injdim{inj\,dim} \opn\rank{rank}
\opn\depth{depth} \opn\sdepth{sdepth} \opn\hdepth{hdepth}
\opn\grade{grade} \opn\height{height} \opn\embdim{emb\,dim}
\opn\codim{codim}  \opn\min{min} \opn\max{max}

\opn\Tr{Tr} \opn\bigrank{big\,rank}
\opn\superheight{superheight}\opn\lcm{lcm}
\opn\trdeg{tr\,deg}
\opn\reg{reg} \opn\lreg{lreg} \opn\ini{in} \opn\lpd{lpd}
\opn\size{size}
\opn\div{div} \opn\Div{Div} \opn\cl{cl} \opn\Cl{Cl}
\opn\Spec{Spec} \opn\Supp{Supp} \opn\supp{supp} \opn\Sing{Sing}
\opn\Ass{Ass} \opn\Min{Min}
\opn\Ann{Ann} \opn\Rad{Rad} \opn\Soc{Soc}
\opn\Im{Im} \opn\Ker{Ker} \opn\Coker{Coker} \opn\Am{Am}
\opn\Hom{Hom} \opn\Tor{Tor} \opn\Ext{Ext} \opn\End{End}
\opn\Aut{Aut} \opn\id{id}  \opn\deg{deg}

\opn\nat{nat}
\opn\pff{pf}
\opn\Pf{Pf} \opn\GL{GL} \opn\SL{SL} \opn\mod{mod} \opn\ord{ord}
\opn\Gin{Gin} \opn\Hilb{Hilb}
\opn\aff{aff} \opn\con{conv} \opn\relint{relint} \opn\st{st}
\opn\lk{lk} \opn\cn{cn} \opn\core{core} \opn\vol{vol}
\opn\link{link} \opn\star{star}
\opn\gr{gr}

\def\pot#1#2{#1[\kern-0.28ex[#2]\kern-0.28ex]}

\opn\Mon{Mon}
\opn\lm{LM}
\opn\lexp{LE}
\opn\lt{LT}
\opn\lc{LC}
\opn\z{\mathbb Z}
\opn\q{\mathbb Q}
\opn\tail{tail}
\opn\quot{Quot}
\opn\L{L}
\opn\nf{NF}
\opn\ecart{ecart}
\opn\sp{spoly}
\opn\smiss{smiss}
\opn\gmiss{gmiss}
\opn\syz{Syz}
\opn\hc{HC}
\opn\h{H}
\opn\hp{HP}
\opn\p{P}
\opn\gp{gcd-poly}
\opn\ssyz{\sp_{syz}}
\opn\gsyz{\gp_{syz}}
\opn\sing{\textsc{Singular}}
\opn\dirlim{\underrightarrow{\lim}}
\opn\inivlim{\underleftarrow{\lim}}
\let\iso=\cong

\let\Dirsum=\bigoplus

\newcommand{\itce}{\item[$\circ$]}

\newcommand{\fracs}[2]{\displaystyle\frac{#1}{#2}}
\newcommand{\Dsum}[2]{\displaystyle\Dirsum_{#1}^{#2}}
\newcommand{\Sum}[2]{\displaystyle\sum_{#1}^{#2}}

\def\Implies{\ifmmode\Longrightarrow \else
        \unskip${}\Longrightarrow{}$\ignorespaces\fi}
\def\implies{\ifmmode\Rightarrow \else
        \unskip${}\Rightarrow{}$\ignorespaces\fi}
\def\iff{\ifmmode\Longleftrightarrow \else
        \unskip${}\Longleftrightarrow{}$\ignorespaces\fi}

\let\:=\colon
\newtheorem{Theorem}{Theorem}[section]
\newtheorem{Lemma}[Theorem]{Lemma}

\newtheorem{Proposition}[Theorem]{Proposition}
\newtheorem{Remark}[Theorem]{Remark}

\newtheorem{Example}[Theorem]{Example}

\newtheorem{Definition}[Theorem]{Definition}
\newtheorem{Algorithm}[Theorem]{Algorithm}

\let\epsilon\varepsilon
\let\phi=\varphi
\let\kappa=\varkappa
\def\qed{\ifhmode\textqed\fi
      \ifmmode\ifinner\quad\qedsymbol\else\dispqed\fi\fi}
\def\textqed{\unskip\nobreak\penalty50
       \hskip2em\hbox{}\nobreak\hfil\qedsymbol
       \parfillskip=0pt \finalhyphendemerits=0}
\def\dispqed{\rlap{\qquad\qedsymbol}}

\opn\dis{dis}
\def\pnt{{\raise0.5mm\hbox{\large\bf.}}}

\opn\Lex{Lex}

\begin{document}

\title{\bf An algorithm to compute the  Hilbert depth}

\author{Adrian Popescu}

\thanks{The  support from the Department of Mathematics of the University of Kaiserslautern is gratefully acknowledged.}

\address{Adrian Popescu, Department of Mathematics, University of Kaiserslautern, Erwin-Schr\"odinger-Str., 67663 Kaiserslautern, Germany}
\email{popescu@mathematik.uni-kl.de}

\maketitle
\begin{abstract}
We give an algorithm which computes the Hilbert depth of a graded module based on a theorem of Uliczka. Partially answering a question of Herzog, we see that the Hilbert depth of a direct sum of modules can be strictly greater than the Hilbert depth of all the summands.

  \vskip 0.4 true cm
 \noindent

\end{abstract}

\section*{Introduction}
\vskip 0.7cm

Let $K$ be a field and $R = K[x_1 \ldots, x_n]$ be the polynomial algebra over $K$ in $n$ variables. On $R$ consider the following two grading structures: the $\z-$grading in which each $x_i$ has degree $1$ and
 the multigraded structure, i.e. the $\z^n-$grading in which each $x_i$ has degree the $i-$th vector $e_i$ of the canonical basis.

After Bruns-Krattenthaler-Uliczka \cite{bku} (see also \cite{Sh}),  a \textbf{Hilbert decomposition} of a $\z-$graded   $R-$module $M$ is a finite family $${\mathcal H} = (R_i,s_i)_{i\in I}$$ in which $s_i\in {\z}$  and $R_i$ is a $\z-$graded $K-$algebra retract of $R$ for each $i\in I$ such that  $$M \iso \Dsum{i\in I}{} R_i(-s_i)$$ as a graded $K-$vector space.

The \textbf{Hilbert depth} of $\mathcal H$ denoted by $\hdepth_1 \mathcal H$ is the depth of the $R-$module $\Dsum{i\in I}{} R_i(-s_i)$. The \textbf{Hilbert  depth} of $M$ is defined as $$\hdepth_1 (M) = \max\{\hdepth_1 \mathcal H\ |\ \textnormal{$\mathcal H$ is a Hilbert decomposition of }M\}.$$ We set $\hdepth_1 (0)=\infty$.

\begin{Theorem} (Uliczka \cite{Uli})\label{th uli}\ $\hdepth_1 (M)=\max\{e\ |\ {(1-t)}^e HP_M(t)\textnormal{ is positive}\}$, where $\hp_M(t)$ is the Hilbert$-$Poincar\'e series of $M$ and a Laurent series in $\z[[t, t^{-1}]]$ is called \textbf{positive} if it has only nonnegative coefficients.
\end{Theorem}

If  $M$ is a multigraded $\z^n-$module, then one can define $\hdepth_n (M)$ as above by considering the $\z^n-$grading instead of the standard one. There exists an algorithm for computing the $\hdepth_n$ of a finitely generated multigraded module $M$ over the standard multigraded polynomial ring $K[x_1,\ldots, x_n]$ in Ichim and Moyano-Fern\'andez's paper \cite{IM} (see also \cite{IZ}).

The main purpose of this paper is to provide an algorithm to compute $\hdepth_1 (M)$, where $M$ is a graded $R-$module (see Algorithm \ref{alg: hdepth}). This is part of the author's Master Thesis \cite{Master}.

A \textbf{Stanley decomposition} (see \cite{Stanley}) of a $\z-$graded (resp. ${\z}^n-$graded)  $R-$module $M$ is a finite family $${\mathcal D} = (R_i,u_i)_{i\in I}$$ in which $u_i$ are homogeneous elements of $M$ and $R_i$ is a graded (resp. ${\z}^n-$graded) $K-$algebra retract of $R$ for each $i\in I$ such that $R_i\cap \Ann(u_i)=0$ and $$M=\Dsum{i\in I}{} R_i u_i$$ as a graded $K-$vector space.

The \textbf{Stanley depth} of $\mathcal D$ denoted by $\sdepth \mathcal D$ is the depth of the $R-$module $\Dsum{i\in I}{} R_iu_i$. The \textbf{Stanley depth} of $M$ is defined as $$\sdepth (M) = \max\{\sdepth \mathcal D\ |\ \textnormal{$\mathcal D$ is a Stanley decomposition of }M\}.$$ We set $\sdepth (0)=\infty$.

We talk about $\sdepth_1(M)$ and  $\sdepth_n(M)$ if we consider the $\z-$grading respectively the ${\z}^n-$grading of $M$. The Hilbert depth of $M$ is greater than the Stanley depth of $M$ and can be strictly greater (an example can be found in \cite{bku}).

Herzog posed the following question (see also \cite[Problem 1.67]{bg}): is $\sdepth_n(R \oplus m) = \sdepth_n (m)$, where $m$ is the maximal ideal in $R$? Since we implemented an algorithm to compute $\hdepth_1$, we have tested whether $\hdepth_1(R \oplus m) = \hdepth_1 (m)$ and as a consequence when $\sdepth_n(R \oplus m) = \sdepth_n (m)$. Proposition \ref{prop: p} says that Herzog's question holds for $n \in \{1, \ldots, 5, 7, 9, 11\}$, but Remark \ref{rem: end} says that for $n = 6$ it holds $\hdepth_1 (R \oplus m) > \hdepth_1 m$, which is a sign that in this case $\sdepth_n (R \oplus m) > \sdepth_n m$ and so Herzog's question could have a negative answer for $n = 6$. This is indeed the case as it was shown later by Ichim and Zarojanu in \cite{IZ}. Meanwhile  Bruns et. al. \cite{BFU} found another algorithm computing $\hdepth_1$ and Chen \cite{C} gave another one in the frame of ideals.

 We owe thanks to Ichim who suggested us this problem and to Uliczka who found a mistake in a previous version of our algorithm.

\vskip 1 cm
\section{hdepth Computation}
\vskip 0.7cm
In this section we introduce an algorithm which computes  $\hdepth_1$ (Algorithm \ref{alg: hdepth}) and prove its correctness (Theorem \ref{th: alg}). In the next section we provide some examples and some results related to \cite[Problem 1.67]{bg}.


\begin{Remark}{\em The algorithm presented in this section is based on Theorem \ref{th uli} and at a first glance it might look trivial. The difficulty lies in the fact that it is not clear how many coefficients of the infinite Laurent series have to be checked for positivity. This paper provides a bound up to which it suffices to check.}\end{Remark}

Recall first \cite[Corollary 4.1.8]{bruns_herzog} the definition of the Hilbert$-$Poincar\'e series of a module $M$ \begin{equation}\label{hp form}\hp_M(t)=\fracs{Q(t)}{{(1-t)}^n}=\fracs{G(t)}{{(1-t)}^{d}}\ ,\end{equation} where $d = \dim M$ and $Q(t),\ G(t) \in \z[t],\ G(1) \neq 0$. In fact, note that $G(1)$ is equal to the multiplicity of the module which is known to be positive.

The algorithm which we construct requires the module $M$ as the input. Actually we only need the $G(t)$ from (\ref{hp form}) and the dimension of $M$.

\begin{Definition}{\em Let $p(t) = \Sum{i=0}{\infty}a_i \cdot t^i \in \z[[t]]$ be a formal power series. By \textnormal{jet$_j$($p$)} we understand the polynomial jet$_j(p) = \Sum{i=0}{j}a_i \cdot
t^i$.}\end{Definition}

\begin{Algorithm}\label{alg: hdepth}{\em We now present the algorithm that computes the $\hdepth_1$ of a $\z-$graded module $\verb"M"$.
The algorithm uses the following procedures which can easily be constructed in any computer algebra system:
\begin{itemize}
\itce \verb"inverse(poly p, int bound)": computes the inverse of a power series \verb"p" till the degree \verb"bound",
\itce \verb"hilbconstruct(module M)": computes the second Hilbert series of the module  \verb"M" -  a way to do this in $\sing$ is to use the already built-in function \verb"hilb(module M, 2)" which returns the list of coefficients of the second Hilbert series and construct the series,
\itce \verb"positive(poly f)": returns \verb"1" if $f$ has all the coefficients nonnegative and \verb"0" else,
\itce \verb"sumcoef(poly f)": returns the sum of the coefficients of \verb"f",
\itce \verb"jet(poly p, int j)": returns the jet$_{\verb"j"}$ \verb"p". This procedure is already implemented in $\sing$,
\itce \verb"dim(module M)": returns the dimension of \verb"M". This procedure is already implemented in $\sing$.
\end{itemize}
 Below we give the algorithm \verb"hdepth(poly g, int dim__M)". Hence in order to compute $\hdepth_1 \verb"M"$, one considers $\verb"g(t) = hilbconstruct( M )"$ and $\verb"dim__M = dim(M)"$.
\begin{algorithm}[H]
\caption{ $\hdepth_1$ (poly g, int dim\_\_M)}

\begin{algorithmic}[1]
\REQUIRE \item[$\circ$]  a polynomial $g(t) \in \z[t]$ (equal to $\hp_M(t)$)
 \item[$\circ$] an integer $dim\_\_M = \dim M$
\ENSURE \item[$\circ$] $\hdepth_1 M$
\vspace{0.4cm}
\IF {positive($g$) = 1}
\RETURN $dim\_\_M$;
\ENDIF
\STATE poly $f = g$;
\STATE int $c$, $d$, $\beta$;
\STATE $\beta$ = $\deg(g)$;
\FOR{$d = dim\_\_M$ \TO $d = 0$}
 \STATE $d = d-1$;
 \STATE $f = \textnormal{jet}(\ g \cdot \textnormal{inverse}({\ (1-t)}^{dim\_\_M-d},\ \beta\ )\ );$
 \IF{positive($f$) = 1}
 \RETURN $d$;
 \ENDIF
 \STATE $c$ = sumcoef($f$);
 \IF{$c < 0$}
 \WHILE{$c < 0$}
 \STATE $\beta = \beta + 1$;
 \STATE $f = \textnormal{jet}(\ g \cdot \textnormal{inverse}(\ {(1-t)}^{dim\_\_M-d},\ \beta\ )\ );$
 \STATE $c$ = sumcoef($f$);
 \ENDWHILE
 \ENDIF
\ENDFOR
\end{algorithmic}
\end{algorithm}

}\end{Algorithm}

\begin{Theorem} \label{th: alg}Given a $\z-$graded module $M$, Algorithm \ref{alg: hdepth} correctly computes \begin{equation}\label{eq: max uliczka}\max \left\{n\ \middle|\ {(1-t)}^n \cdot \hp_M(t) \textnormal{ is positive }\right\}\end{equation} where $\hp_M(t) = \fracs{G(t)}{{(1-t)}^{\dim M}}$ is the Hilbert-Poincar\'e series of $M$. Hence, by Theorem \ref{th uli}, the algorithm computes the Hilbert depth of a module $M$ for $g = G(t)$ and $dim\_\_M = \dim M$.
\end{Theorem}
\begin{proof}Note that $G(1)$ is the multiplicity of the module $M$ and hence $G(1) > 0$.

Assume that $M \neq 0$. Denote the bound $\beta$ at the end of the loop where $d = i$ by $\beta_i$. In order to prove this theorem one has to show the following two claims:
 \begin{itemize}\itce the maximum from (\ref{eq: max uliczka}) does not exceed $\dim M$,
 \itce after the bound $\beta_i$ degree, the coefficients are nonnegative.
\end{itemize}
For the first part consider $G(t) = \Sum{\mu=0}{g}a_\mu \cdot t^\mu$. Note that $$(1-t)^{\dim M +1} \cdot \hp_M(t) = (1-t) \cdot G(t) = a_0 + (a_1 - a_0)\cdot t + \ldots + (a_g - a_{g-1}) \cdot t^g - a_g \cdot t^{g+1}.$$ If all coefficients would be nonnegative, we would obtain $$0 \geq a_g \geq a_{g-1} \geq a_{g-2} \geq \ldots \geq a_2 \geq a_1 \geq a_0 \geq 0$$ which implies that $G(t) = 0$. This will lead to a contradiction with $M \neq 0$. The same holds for $(1-t)^{\dim M + \alpha} \cdot \hp_M(t)$ by considering $(1-t)^{\dim M  + \alpha - 1} \cdot \hp_M(t)$ instead of $G(t)$, where $\alpha \geq 0$. Thus the maximum from (\ref{eq: max uliczka}) is smaller or equal than $\dim M$.

Note that if $G(t)$ already has all the coefficients nonnegative, then the algorithm stops by returning $\dim M$, and the result is correct since in this case $\hdepth_1 M = \dim M$.

For the second part we need to show that at each step $i$ the coefficient of the term of order $\beta_i$ in $\fracs{G(t)}{{(1-t)}^{\dim M - i}}$  is nonnegative and the coefficients of the terms of higher order are increasing (and hence nonnegative). Apply induction on $i$. For the first step, $d = \dim M - 1$,  $f = \fracs{G(t)}{(1-t)}$ and all the coefficients of the terms of order $\geq \beta_{\dim M - 1} = \deg G(t)$ are equal to the sum of the coefficients $G(1) > 0$. For the general step $i$, assume that at the beginning of loop $d = i$, we started with  $\fracs{G(t)}{{(1-t)}^{\dim M - i}} = \Sum{\mu=0}{\infty} b_\mu \cdot t^\mu$ which satisfied all the desired properties  by induction: the bound $\beta_{i}$ was increased (if required),  such that the coefficient sum $c_i := \Sum{\mu=0}{\beta_{i}} b_\mu > 0$ and all coefficients of higher order terms are nonnegative, i.e. $b_\mu \geq 0$ for $\mu \geq \beta_{i-1}$. We now consider the next step, $d = i-1$, and compute the new $f$ as in line 9 of the algorithm. In order to check that the coefficients of the terms of order higher than the bound $\beta_i$ are nonnegative. We have: $$\fracs{G(t)}{{(1-t)}^{\dim M - (i-1)}} = \overbrace{b_0 + (b_0 + b_1)\cdot t + \ldots + \underbrace{\left(\Sum{\mu=0}{\beta_{i}}b_\mu\right)}_{c_i > 0} \cdot t^{\beta_{i}}}^{= \textnormal{ jet}_{\beta_{i}}} + (c_i + b_{\beta_{i}+1}) \cdot t^{\beta_{i}+1} + \ldots$$

By induction, $0 < b_{\beta_{i}} \leq b_{\beta_{i}+1} \leq b_{\beta_{i}+2} \leq \ldots$ and since $c_i > 0$ we obtain $c_i + b_{\beta_{i} + \nu} > 0$ for $\nu \geq 0$.

The termination of the algorithm is trivial since we know that in the last loop we would consider $\fracs{G(t)}{{(1-t)}^{\dim M}} = \hp_M(t)$ which is positive by the definition, and hence it will return $\hdepth_1 M = 0$.
\hfill\ \end{proof}

\begin{Remark}{\em The maximum from the statement  of \cite[Theorem 3.2]{Uli} (see here Theorem \ref{th uli}) is always smaller than $\dim M$. This was not shown in Uliczka's proof and it has to be proved in Theorem \ref{th: alg}.}
\end{Remark}

\vskip 1cm
\section{Computational Experiments}
\vskip 1cm
The following examples illustrate the usage of the implementation of the algorithm in $\sing$, which can be found in the Appendix. Note that in the outputs we print exactly the jet we considered in our computations followed by ``\verb"+..."".
\begin{Example}{\em Consider the ring $\q[x, y_1, \ldots, y_5]$ and consider the ideal $I = (x)\cap(y_1, \ldots, y_5)$.
\begin{verbatim}
ring R=0,(x,y(1..5)),ds;
ideal i=intersect(x,ideal(y(1..5)));
module m=i;
"dim M = ",dim(m);
//  dim M = 5
hdepth( hilbconstruct( m ), dim(m) );
//  G(t)= 1+t-4t2+6t3-4t4+t5
//  G(t)/(1-t)^ 1 = 1+2t-2t2+4t3+t5 +...
//  G(t)/(1-t)^ 2 = 1+3t+t2+5t3+5t4+6t5 +...
//  hdepth= 3\end{verbatim}}
\end{Example}

\begin{Example}{\em Consider a module $M$ for which $\hp_M(t) = \fracs{2-3t-2t^2+2t^3+4t^4}{{(1-t)}^{\dim M}}$. Denote by $\verb"dim__M"$ the dimension of $M$.
\begin{verbatim}
ring R = 0, t, ds;
poly g = 2-3*t-2*t^2+2*t^3+4*t^4;
hdepth( g, dim__M);
// G(t)= 2-3t-2t2+2t3+4t4
// G(t)/(1-t)^ 1 = 2-t-3t2-t3+3t4+3t5 +...
// G(t)/(1-t)^ 2 = 2+t-2t2-3t3+3t5 +...
// G(t)/(1-t)^ 3 = 2+3t+t2-2t3-2t4+t5 +...
// G(t)/(1-t)^ 4 = 2+5t+6t2+4t3+2t4+3t5 +...
\end{verbatim} Hence, it results $\hdepth_1 M = \dim M - 4$.

As seen in the proof, we had to increase our bound if the coefficient sum was $\leq 0$. Note that in this example, the coefficient sum of jet$_4$$\left(\fracs{G(t)}{(1-t)}\right)$ is zero and thus we increase the bound to $5$ (the coefficient sum of the jet$_5$ will be equal to $3 > 0$).}
\end{Example}


\begin{Example}{\em
Consider $R = K[x_1, \ldots, x_n]$ for $n \in \{4,5,\ldots, 19\}$ and $m$ the maximal ideal. We computed
$\hdepth_1 m$, $\hdepth_1 (R \oplus m)$, $\ldots $, $\hdepth_1 (R^6\oplus m)$ and $\hdepth_1 (R^{100} \oplus m)$. We obtain the following results: \vskip 0.7cm
\begin{figure}[H]
\label{fig: figure}
\begin{center}
  \begin{tabular}{ r || c | c | c | c | c | c | c | c | c | c | c | c | c | c | c | c | }

    n & 4 & 5 & 6 & 7 & 8 & 9 & 10 & 11 & 12 & 13 &14 &15 &16 &17 &18 &19 \\ \hline \hline
    $\hdepth_1 (m)$ & 2 & 3&3 &4 &4  &5 &5 &6 &6  &7 &7 &8 &8 &9 &9 &10 \\ \hline
    $\hdepth_1 (R \oplus m)$ &2 &3 &4 &4 &5  &5 &6 &6 &7  &8 &8 &9 &9 &10 &11 &11 \\ \hline
    $\hdepth_1 (R^2 \oplus m)$ &3 &3 &4 &4 &5  &6 &6 &7 &8  &8 &9 &10 &10 &11 &11 &12 \\ \hline
    $\hdepth_1 (R^3 \oplus m)$ &3 &3 &4 &5 &5  &6 &7 &7 &8  &9 &9 &10 &10 &11 &12 &12 \\ \hline
    $\hdepth_1 (R^4 \oplus m)$ &3 &3 &4 &5 &6  &6 &7 &8 &8  &9 &9 &10 &11 &11 &12 &12 \\ \hline
    $\hdepth_1 (R^5 \oplus m)$ &3 &4 &4 &5 &6  &6 &7 &8 &8  &9 &10 &10 &11 &11 &12 &13 \\ \hline
    $\hdepth_1 (R^6 \oplus m)$ &3 &4 &4 &5 &6  &7 &7 &8 &8  &9 &10 &10 &11 &11 &12 &13 \\ \hline
    $\hdepth_1 (R^{100} \oplus m)$ &3 &4 &5 &6 &7  &8 &8 &9 &10  &11 &11 &12 &13 &13 &14 &15 \\ \hline
    \hline
  \end{tabular}
\end{center}
\caption{}
\end{figure}}
\end{Example}
\begin{Remark}\label{rem: end}{\em Note that for $n = 6$ we have $\hdepth_1 (R \oplus m) = 4 > 3 = \hdepth_1 m$. This is a sign that in this case $\sdepth_n (R \oplus m) > \sdepth_n (m)$ and so Herzogs's question could have a negative answer for $n = 6$. The difference $\hdepth_1 (R \oplus m) - \hdepth_1 m$ can be $> 1$ as one can see for $n = 18$.

Note that $\hdepth_1 (R^s \oplus m) - \hdepth_1 m$ increases when $s$ and $n$ increase. For example $\hdepth_1 (R^{100} \oplus m) - \hdepth_1 m = 5$ for $s = 100$ and $n=19$. }
\end{Remark}
\begin{Lemma}\label{lemma: lemma}{\em Let $n \in \mathbb{N}$ be such that $\hdepth_1 m = \hdepth_1 (R \oplus m)$. Then $\sdepth_n m = \sdepth_n (R \oplus m)$.}
\end{Lemma}

\begin{proof}By \cite{bku} and \cite{BIRO} we have $\hdepth_1 m = \left\lceil\fracs{n}{2}\right\rceil = \sdepth_n m$. It is enough to see that the following inequalities hold: \\
$\hdepth_1 m = \sdepth_n m \leq \sdepth_n (R \oplus m) \leq \hdepth_n (R \oplus m) \leq \hdepth_1 (R \oplus m).$

\ \hfill\ \end{proof}

\begin{Proposition}\label{prop: p}{\em If $n \in \{1,\ldots, 5,7,9,11\}$ then $\sdepth_n m = \sdepth_n (R \oplus m)$, that is Herzog's question has a positive answer.}
\end{Proposition}
\begin{proof}Note that $\hdepth_1 m = \hdepth_1 (R \oplus m)$ for $n$ as above and apply Lemma \ref{lemma: lemma}.
\hfill\ \end{proof}

\newpage
\section*{Appendix}

As stated before, Algorithm \ref{alg: hdepth} was implemented as a procedure for the computer algebra system $\sing$ \cite{singularbook}. This procedure was used in order to obtain the results from Figure 1. The additional procedures which have been used were defined in Algorithm \ref{alg: hdepth}. In addition, we printed some information which we find useful for understanding the algorithm.

\lstset{language=C++,basicstyle=\footnotesize,frame=single}
\begin{lstlisting}
proc hdepth(poly g, int dim__M)
{
        int d;
	ring T = 0,t,ds;
		"G(t)=",g;
	if(positiv(g)==1)
		{return("hdepth=",dim__M);}
	poly f=g;
	number ag;
	int c1;
	int bound;
	bound = deg(g);
	for(d = dim__M; d>=0; d--)
	{
	   f = jet( g*inverse( (1-t)^(dim__M-d),bound ) , bound );
            if(positiv(f) == 1)
            {
		"G(t)/(1-t)^",dim__M-d,"=",f,"+...";
        	"hdepth=",d;
        	return();
            }
            c1=sumcoef(f);
	   if(c1<=0)
            {
       	        while( c1<0 )
                {
                bound = bound + 1;
                f =  jet( g*inverse( (1-t)^(dim__M-d),bound ) , bound );
                c1 = sumcoef(f);
                }
                "G(t)/(1-t)^",dim__M-d,"=",g,"+...";
            }
        }
}
\end{lstlisting}

\end{document}